\documentclass[10pt]{amsart}

\usepackage{amssymb,amsmath,amsthm,newlfont}
\usepackage[shortlabels]{enumitem}

\numberwithin{equation}{section}

\theoremstyle{plain}
\newtheorem{theorem}{Theorem}[section]
\newtheorem{proposition}[theorem]{Proposition}
\newtheorem{corollary}[theorem]{Corollary}
\newtheorem{lemma}[theorem]{Lemma}

\theoremstyle{remark}
\newtheorem*{remark}{Remark}
\newtheorem*{remarks}{Remarks}

\newcommand{\CC}{{\mathbb C}}
\newcommand{\DD}{{\mathbb D}}

\newcommand{\NN}{{\mathbb N}}

\newcommand{\TT}{{\mathbb T}}
\newcommand{\ZZ}{{\mathbb Z}}
\newcommand{\cB}{{\cal B}}
\newcommand{\cD}{{\cal D}}
\newcommand{\cH}{{\cal H}}
\newcommand{\cK}{{\cal K}}

\DeclareMathOperator{\dist}{dist}
\DeclareMathOperator{\spn}{span}
\DeclareMathOperator{\hol}{Hol}
\DeclareMathOperator{\bmoa}{BMOA}
\DeclareMathOperator{\vmoa}{VMOA}

\renewcommand{\hat}{\widehat}

\begin{document}

\title{Summability and duality}

\author{Soumitra Ghara}
\address{D\'epartement de math\'ematiques et de statistique, Universit\'e Laval,
Qu\'ebec City (Qu\'ebec),  Canada G1V 0A6.}
\email{ghara90@gmail.com}

\author{Javad Mashreghi}
\address{D\'epartement de math\'ematiques et de statistique, Universit\'e Laval,
Qu\'ebec City (Qu\'ebec),  Canada G1V 0A6.}
\email{javad.mashreghi@mat.ulaval.ca}

\author{Thomas Ransford}
\address{D\'epartement de math\'ematiques et de statistique, Universit\'e Laval,
Qu\'ebec City (Qu\'ebec),  Canada G1V 0A6.}
\email[Corresponding author]{thomas.ransford@mat.ulaval.ca}

\thanks{Ghara supported by a Fields--Laval postdoctoral fellowship. Mashreghi supported by an NSERC Discovery Grant. Ransford  supported by grants from NSERC and the Canada Research Chairs program.}

\begin{abstract}
We formalize the observation that the same summability methods converge in a Banach space $X$ and its dual $X^*$.
At the same time we determine conditions under which these methods converge in the weak and weak*-topologies on  $X$ and $X^*$ respectively. We also derive a general  limitation theorem, which yields a necessary condition for the convergence of a summability method in $X$.
These  results are then illustrated by applications to a wide variety of function spaces, including spaces of continuous functions, Lebesgue spaces, the disk algebra, Hardy and Bergman spaces, the BMOA space, the Bloch space, and de Branges--Rovnyak spaces. Our approach shows that all these applications flow from just two abstract theorems.
\end{abstract}


\makeatletter
\@namedef{subjclassname@2020}{\textup{2020} Mathematics Subject Classification}
\makeatother


\maketitle

\section{Introduction}\label{S:intro}

Let $X$ be a Banach space of holomorphic functions on the open unit disk~$\DD$,
and suppose that $X$ contains the polynomials.
Every function $f\in X$ has a Taylor expansion $f(z)=\sum_{j\ge0}a_jz^j$,
which converges locally uniformly on $\DD$ to $f(z)$.
However, it can happen that the series fails to converge to $f$ in the norm of $X$.
This is the case, for example, whenever polynomials are not dense in $X$, 
but it may occur even when they are dense.
Here is a short list of examples illustrating various possibilities, ranging from `best' to `worst'.

\begin{enumerate}[(1)]
\item If $X$ is the Hardy space $H^2$, then the Taylor series of $f$ does converge to $f$ in the norm of $X$.
The same is true if $X$ is the Dirichlet space or the Bergman space.
\item If $X$ is the disk algebra $A(\DD)$, then the Taylor series of $f$ may fail to converge in the norm of $X$
(du Bois-Reymond's example),
but its Ces\`aro means do converge in norm (Fej\'er's theorem).
\item If $X$ is a de Branges--Rovnyak space $\cH(b)$, then, for certain choices of $b$ and $f$, the Ces\`aro means may fail to converge in norm, though polynomials are still dense in $X$ (see \cite{EFKMR16}).
\item If $X$ fails to have the bounded approximation property, then no lower-triangular summation
method can converge in norm for every function, even though polynomials may still be  dense in $X$
(see \cite{MR19}). 
\end{enumerate}

In this article we are mainly interested in cases like (2) and (3), where 
some summability methods work and others do not, and the problem is to determine the range of methods 
that do work. Our starting point is the fact that, typically, the same methods tend to work in $X$ and 
in its dual $X^*$. This is because the convergence of a summability method often boils down to whether a certain sequence of summation operators  is uniformly bounded in norm, and an operator has the same norm as its adjoint.
We formalize this idea, at the same time linking it to weak- and weak*-convergence in $X$ and $X^*$ respectively. We also derive a general limitation theorem, namely a necessary condition for the convergence of summability method in a given Banach space. The proofs of these results are carried out in two steps: 
in \S\ref{S:operator} we establish a general operator-theoretic result, which is then used 
in \S\ref{S:summability} to derive the abstract summability theorems.

The rest of the article is devoted to various examples and applications of these  results.
In \S\ref{S:continuous} we consider  continuous-function and Lebesgue spaces, as well as the disk algebra.
In \S\ref{S:HardyBergman} we treat the Hardy and Bergman spaces and their relatives, $\bmoa$ and the Bloch space.
Finally, in \S\ref{S:Hilbert}, we consider reproducing kernel Hilbert spaces of holomorphic functions, and, in particular, de Branges--Rovnyak spaces.
Some of the applications are already known, others are slight generalizations of known results, and some are completely new.
Our approach shows that they all flow from just two abstract theorems.

\section{Operator theory}\label{S:operator}

In what follows, $X$ is a real or complex Banach space. 
We write $X^*$ for the dual space of $X$, 
and $\langle\cdot,\cdot\rangle$ for the duality pairing between $X$ and $X^*$.
We use $w$ and $w^*$ to denote the weak and weak* topologies on $X$ and $X^*$ respectively.
Lastly, given a bounded operator $T$ on $X$, 
we write $T^*$ for the adjoint operator on $X^*$, defined by the relation
\[
\langle x,T^*\phi\rangle=\langle Tx,\phi\rangle \quad(x\in X,~\phi\in X^*).
\]

The purpose of this section is to establish the following result.

\begin{theorem}\label{T:norm-norm}
Let $(T_n)_{n\ge0}$ be a sequence of bounded, finite-rank operators on $X$
such that
\begin{equation}\label{E:norm-norm}
T_nT_m(X)\subset T_m(X) 
\quad\text{and}\quad
T_n^*T_m^*(X^*)\subset T_m^*(X^*)
\qquad(m,n\ge0).
\end{equation}
Let  
\[
Y:=\overline{\spn(\cup_{m\ge0}T_m(X))}
\quad\text{and}\quad
Z:=\overline{\spn(\cup_{m\ge0}T_m^*(X^*))},
\]
where the closures are taken in the norm topologies of $X$ and $X^*$ respectively.
Then the following statements are equivalent:
\begin{enumerate}[\normalfont\rm(i)]
\item $T_n x\to x$ in $(X,w)$ for all $x\in X$;
\item $T_n x\to x$ in $(X,\|\cdot\|)$ for all $x\in X$;
\item $T_n^*\phi\to\phi$ in $(X^*,w^*)$ for all $\phi\in X^*$;
\item $T_n^*\phi\to \phi$ in $(X^*,w^*)$ for all $\phi\in Z$, and $Z$ is $w^*$-sequentially dense in $X^*$ and $Y=X$;
\item $T_n^*\phi\to \phi$ in $(X^*,\|\cdot\|)$ for all $\phi\in Z$, and  $Z$ is $w^*$-sequentially dense in $X^*$
and $Y=X$.
\end{enumerate}
If, further, $X$ is reflexive, then these are equivalent to:
\begin{enumerate}[\normalfont\rm(vi)]
\item $T_n^*\phi\to\phi$ in $(X^*,\|\cdot\|)$ for all $\phi\in X^*$.
\end{enumerate}
\end{theorem}

We shall prove this result via a series of lemmas, beginning with a very simple one.

\begin{lemma}\label{L:weak-weak*}
Let $(T_n)_{n\ge0}$ be a sequence of bounded operators on $X$.
The following statements are equivalent:
\begin{enumerate}[\normalfont(i)]
\item $T_n x\to x$ in $(X,w)$ for all $x\in X$;
\item $T_n^*\phi\to \phi$ in $(X^*,w^*)$ for all $\phi\in X^*$.
\end{enumerate}
\end{lemma}

\begin{proof}
We have
\begin{align*}
T_n x \overset{w}\longrightarrow x \quad\forall x\in X
&\iff \langle T_nx,\phi\rangle\to\langle x,\phi\rangle \quad\forall \phi\in X^*,\forall x\in X\\
&\iff \langle x,T_n^*\phi\rangle\to\langle x,\phi\rangle \quad\forall x\in X, \forall \phi\in X^*\\
&\iff T_n^*\phi  \overset{w^*}\longrightarrow \phi \quad\forall \phi\in X^*.\qedhere
\end{align*}
\end{proof}

We next establish a similar result relating weak convergence in $X$ to norm convergence. 
In order to obtain an equivalence, we need to impose some conditions on the operators $T_n$. 

\begin{lemma}\label{L:weak-norm}
Let $(T_n)_{n\ge0}$ be a sequence of bounded, finite-rank operators on $X$ such that
\begin{equation}\label{E:weak-norm}
T_nT_m(X)\subset T_m(X) \qquad(m,n\ge0).
\end{equation}
Then the  following statements are equivalent:
\begin{enumerate}[\normalfont(i)]
\item $T_n x\to x$ in $(X,w)$ for all $x\in X$;
\item $T_nx \to x$ in $(X,\|\cdot\|)$ for all $x\in X$.
\end{enumerate}
\end{lemma}

\begin{remarks}
(1) Lemma~\ref{L:weak-norm} fails without the assumption `finite-rank'.
For example, if $X:=\ell^2(\ZZ^+)$ and $T_n:=I+S^n$, 
where $S$ is the unilateral shift on $\ell^2(\ZZ^+)$, then $T_nT_m=T_mT_n$ for all $m,n$, 
so \eqref{E:weak-norm} holds, and
\[
T_nx-x=S^n x\overset{w}\longrightarrow0 \quad\forall x\in \ell^2(\ZZ^+),
\]
but
\[
\|T_ne_0-e_0\|_2=\|e_{n}\|_2\not\to0.
\]

(2) Lemma~\ref{L:weak-norm} also fails without the assumption \eqref{E:weak-norm}. 
For example, if $X=\ell^2(\ZZ^+)$ and 
$T_n:=\sum_{j=0}^n(e_j\otimes e_j)+(e_n\otimes e_0)$, 
then each $T_n$ is a bounded, finite-rank operator, and
\[
T_n x-x=-\sum_{j>n}\langle x,e_j\rangle e_j+\langle x,e_0\rangle e_n
\overset{w}\longrightarrow0 \quad \forall x\in\ell^2(\ZZ^+),
\]
but
\[
\|T_ne_0-e_0\|_2=\|e_n\|_2\not\to0.
\]
Note that, in this example, if $m<n$, 
then $T_m(X)=\spn\{e_0,e_1,\dots,e_m\}$, 
while $T_nT_m(X)=\spn\{e_0+e_n,e_1,\dots,e_m\}$.
\end{remarks}

\begin{proof}[Proof of Lemma~\ref{L:weak-norm}]
It is enough to prove that (i) $\Rightarrow$ (ii), since the reverse implication is obvious. 
Suppose then that $T_n x\to x$ weakly for all $x\in X$. 
We need to show that it also converges in norm. 
This will be done in four steps.

The first step is to show that $\|T_n x-x\|\to0$ if 
$x\in \spn(\cup_mT_m(X))$.
Fix $m$ and let $\psi$ be a continuous linear functional on $T_m(X)$.
By the Hahn--Banach theorem, 
we can extend $\psi$ to a continuous linear functional $\phi$ on the whole of $X$. 
Therefore, for all $x\in T_m(X)$, we have
\[
\langle T_n x,\psi\rangle
=\langle T_n x,\phi\rangle
\to\langle x, \phi\rangle
=\langle x,\psi\rangle.
\]
This shows that $T_n x\to x$ in the weakly in $T_m(X)$. As $\dim(T_m(X))<\infty$,
the weak topology and norm topology coincide, so $\|T_nx-x\|\to0$ for all $x\in T_m(X)$.
Finally, by linearity, it follows that $\|T_nx-x\|\to0$ for all $x\in  \spn(\cup_mT_m(X))$,
as claimed.

The second step is to show that $\spn(\cup_mT_m(X))$ is norm-dense in $X$.
Suppose the contrary. Then, by the Hahn--Banach theorem, 
there exists $\phi\in X^*\setminus\{0\}$ such that
$\phi=0$ on $\spn(\cup_mT_m(X))$. 
For all $m\ge1$, we have
\[
\langle x,T_m^*\phi\rangle=\langle T_mx,\phi\rangle=0 \quad(x\in X),
\]
so $T_m^*\phi=0$. Also, from Lemma~\ref{L:weak-weak*},
we know that  $T_m^*\phi\to\phi$ in $(X^*,w^*)$.  Hence $\phi=0$. 
This contradicts  the choice of $\phi$. 
We conclude that, as claimed, $\spn(\cup_mT_m(X))$ is norm-dense in $X$.

The third step is to show that $\sup_n\|T_n\|<\infty$. For each $x\in X$, the sequence $(T_nx)$ converges weakly, so it is weakly bounded. By the Banach--Steinhaus theorem, it is also norm-bounded, i.e.,
$\sup_n\|T_n x\|<\infty$. As this holds for each $x\in X$, a second application of Banach--Steinhaus shows that $\sup_n\|T_n\|<\infty$, as claimed.

The fourth and final step is to show that $\|T_nx-x\|\to0$ for all $x\in X$.
Let $x\in X$ and let $\epsilon>0$. By steps~2 and 3,
there exists $x_0\in\spn(\cup_mT_m(X))$ such that $\|x-x_0\|<\epsilon/(1+\sup_n\|T_n\|)$.
By step~1, there exists $n_0$ such that $\|T_nx_0-x_0\|<\epsilon$ for all $n\ge n_0$.
Then, if $n\ge n_0$, we have
\[
\|T_nx -x\|\le \|T_n(x-x_0)\|+\|T_n x_0-x_0\|+\|x_0-x\|<3\epsilon.
\]
Thus $\|T_nx-x\|\to0$, as was to be proved.
\end{proof}

If $X$ is reflexive, then we may interchange the roles of $X$ and $X^*$ in
Lemma~\ref{L:weak-norm}, and deduce the following corollary.

\begin{corollary}\label{C:weak-norm}
Suppose that $X$ is reflexive.
Let $(T_n)_{n\ge0}$ be a sequence of bounded, finite-rank operators on $X$ such that
\begin{equation}\label{E:weak*-norm}
T_n^*T_m^*(X^*)\subset T_m^*(X^*)
\qquad(m,n\ge0).
\end{equation}
Then the  following statements are equivalent:
\begin{enumerate}[\normalfont(i)]
\item $T_n^*\phi\to \phi$ in $(X^*,w^*)$ for all $\phi\in X^*$;
\item $T_n^*\phi \to \phi$ in $(X^*,\|\cdot\|)$  for all $\phi\in X^*$.
\end{enumerate}
\end{corollary}

\begin{remark}
If $X$ is not reflexive, then Corollary~\ref{C:weak-norm} may break down.
For example, let $X:=\ell^1(\ZZ^+)$ and let $T_n:\ell^1(\ZZ^+)\to\ell^1(\ZZ^+)$ be the projection onto the first $n$ coordinates. Its adjoint $T_n^*:\ell^\infty(\ZZ^+)\to\ell^\infty(\ZZ^+)$ is also the projection onto the first $n$ coordinates.
The sequences $(T_n)$ and $(T_n^*)$ satisfy 
\eqref{E:weak-norm} and \eqref{E:weak*-norm} respectively.
Also $\|T_nx-x\|_1\to0$ for all $x\in\ell^1(\ZZ^+)$, 
so by Lemmas~\ref{L:weak-weak*} and \ref{L:weak-norm}
we have $T_n^*\phi\to\phi$ weak* for all $\phi\in\ell^\infty(\ZZ^+)$.
However, if $\phi_0:=(1,1,1,\dots)$, then
$\|T_n^*\phi_0-\phi_0\|_\infty\not\to0$.
Note that, in this example, the norm-closure of $\spn\cup_n T_n^*(X^*)$ is $c_0$.
\end{remark}

Here is a version of Corollary~\ref{C:weak-norm} valid for all $X$, not necessarily reflexive.

\begin{lemma}\label{L:weak*-norm}
Let  $(T_n)_{n\ge0}$ be a
sequence of bounded, finite-rank operators on $X$ such that
\begin{equation}\label{E:weak*-norm2}
T_n(T_m(X))\subset T_m(X) 
\quad\text{and}\quad
T_n^*(T_m^*(X^*))\subset T_m^*(X^*) \quad (m,n\ge0).
\end{equation}
Let  
\[
Y:=\overline{\spn(\cup_{m\ge0}T_m(X))}
\quad\text{and}\quad
Z:=\overline{\spn(\cup_{m\ge0}T_m^*(X^*))},
\]
where the closures are taken in the norm topologies of $X$ and $X^*$ respectively.
Then the following statements are equivalent:
\begin{enumerate}[\normalfont\rm(i)]
\item $T_n^*\phi\to \phi$ in $(X^*,w^*)$ for all $\phi\in X^*$;
\item $T_n^*\phi\to \phi$ in $(X^*,w^*)$ for all $\phi\in Z$, and $Z$ is $w^*$-sequentially dense in $X^*$ and $Y=X$;
\item $T_n^*\phi\to \phi$ in $(X^*,\|\cdot\|)$ for all $\phi\in Z$, and $Z$ is $w^*$-sequentially dense in $X^*$ and $Y=X$.
\end{enumerate}
\end{lemma}

For the proof, we need a further lemma.

\begin{lemma}\label{L:ball}
Let $Z$ be a subspace of $X^*$ that is $w^*$-sequentially dense in $X^*$.
Then the $w^*$-closure of the unit ball of $Z$ contains a positive multiple of the unit ball of~$X^*$.
\end{lemma}

\begin{proof}
Let $C$ be the closure in $(X^*,w^*)$ of the unit ball of $Z$.
Given $\phi\in X^*$, there exists a sequence $(\phi_n)$ in $Z$ such that $\phi_n$ is $w^*$-convergent to $\phi$.
By the Banach--Steinhaus theorem, 
since  $(\phi_n)$ is $w^*$-bounded, it is norm-bounded. Hence there exists an integer $m\ge1$
such that $\phi\in mC$. 
Thus we have $\cup_{m\ge1}mC=X^*$. 
Each set $mC$ is $w^*$-closed in $X^*$, 
so it is certainly norm-closed. 
We may therefore apply Baire's theorem to deduce that 
there exists $m_0$ such that $m_0C$ has non-empty norm-interior. 
As $m_0C$ is a convex, symmetric set, 
it follows that $0$ belongs to the norm interior of $m_0C$. 
In other words, $C$ contains a ball around $0$.
\end{proof}

\begin{proof}[Proof of Lemma~\ref{L:weak*-norm}]

[(i)$\Rightarrow$(ii)]: Suppose that (i) holds.
Then it is obvious that $T_n^*\phi\to\phi$ for all $\phi\in Z$,
and also that $Z$ is $w^*$-sequentially dense in $X^*$.
Finally, by Lemmas~\ref{L:weak-weak*} and \ref{L:weak-norm},
(i) implies that $T_nx\to x$ in norm for all $x\in X$,
and this entails that $Y=X$.

\smallskip

[(ii)$\Rightarrow$(iii)]: Suppose that (ii) holds.
By (ii), we have $T_n^*\phi\to\phi$ in $(X^*,w^*)$ for each $\phi\in Z$.
By the Banach--Steinhaus theorem, applied to the sequence $(T_n^*|_Z)$ on $Z$,
we have $\sup_n\|T_n^*|_Z\|<\infty$. 

Let $m\ge0$. Then $T_n^*\phi\to\phi$ in $(X^*,w^*)$ for all
$\phi\in T_m^*(X^*)$. Since $\dim(T_m^*(X^*))<\infty$,
it follows that $T_n^*\phi\to\phi$ in norm for all
$\phi\in T_m^*(X^*)$. As this holds for each $m\ge0$,
we deduce that $T_n^*\phi\to\phi$ in norm for all
$\phi\in \spn(\cup_{m\ge0}T_m^*(X^*))$.
Lastly, as $\sup_n\|T_n^*|_Z\|<\infty$,
it follows that $T_n^*\phi\to\phi$ in norm for all
$\phi\in Z$.

\smallskip

[(iii)$\Rightarrow$(i)]: Suppose that (iii) holds.
As noted above, $\sup_n\|T_n^*|_Z\|<\infty$. 
By (iii) and Lemma~\ref{L:ball},
the unit ball of $Z$ is $w^*$-dense in a ball of radius $r>0$ in $X^*$.
It follows that, for every operator $T$ on $X$, we have $\|T^*\|\le \|T^*|_Z\|/r$. Combining these facts, 
we deduce that $K:=\sup_n\|T_n^*\|<\infty$.

Let $m\ge0$ and let $Z_m:=\{\phi|_{T_m(X)}:\phi\in Z\}$. 
Since $Z$ is $w^*$-sequentially dense in $X^*$
and $\dim T_m(X)<\infty$,  it follows that
$Z_m$ is norm-dense in $T_m(X)^*$. 
Let $x\in T_m(X)$. Let $\psi\in T_m(X)^*$ and let $\epsilon>0$. Then there exists $\phi\in Z$ with $\|\psi-\phi|_{T_m(X)}\|<\epsilon$. Since $\|T_n^*\phi-\phi\|\to0$, there exists $N$ such that
\[
n\ge N 
\quad\Rightarrow \quad
\|T_n^*\phi-\phi\|<\epsilon.
\]
Then
\[
\langle T_n x-x,\,\psi\rangle=\langle x,T_n^*\phi-\phi\rangle +\langle T_nx-x,\, \psi-\phi|_{T_m(X)}\rangle,
\]
so, for all $n\ge N$,
\[
|\langle T_n x-x,\,\psi\rangle|\le \|x\|\|T_n^*\phi-\phi\| +(K+1)\|x\|\| \psi-\phi|_{T_m(X)}\|\le  (\|x\|+K+1)\epsilon.
\]
Thus $T_nx\to x$ weakly for all $x\in T_m(X)$. As $\dim T_m(X)<\infty$, it follows that
$T_n x\to x$ in norm  for all $x\in T_m(X)$. 
As this holds for all $m\ge0$,
we deduce that $T_nx\to x$ in norm for all $x\in \spn(\cup_m T_m(X))$.
Since $\sup_n\|T_n\|=\sup_n\|T_n^*\|=K<\infty$, it follows that
$T_nx\to x$ for all $x\in Y$. 

By (iii), we have $Y=X$. Thus $T_nx\to x$ in norm for all $x\in X$.
By Lemmas~\ref{L:weak-weak*} and \ref{L:weak-norm},
this implies $T_n^*\phi\to\phi$ in $(X^*,w^*)$ for all $\phi\in X^*$.
\end{proof}

\begin{remark}
Although statements (i)--(iii) are all about the adjoint operators $T_n^*$,
we nonetheless need the invariance assumption that
$T_n(T_m(X))\subset T_m(X)$ in \eqref{E:weak*-norm2}.
This assumption is used in the proof of the implication 
(iii)$\Rightarrow$(i),
and the result is actually false without this assumption.
Here is a counterexample. 

Let $X:=\ell^1(\NN)$. For $n\ge1$ define $T_n:\ell^1(\NN)\to\ell^1(\NN)$ by
\[
T_n:=P_{2n}(S^n+I),
\]
where $S:\ell^1(\NN)\to\ell^1(\NN)$ is the unilateral shift
and $P_{2n}:\ell^1(\NN)\to\ell^1(\NN)$ is the projection onto the first $2n$ coordinates.
Taking adjoints, we have $T_n^*:\ell^\infty(\NN)\to\ell^\infty(\NN)$ given by
$T_n^*=(S^{*n}+I)P_{2n}$.
The following properties are easily verified:
\begin{itemize}
\item $T_m(X)=\spn\{e_1,\dots,e_{2m}\}$ for all $m\ge1$.
\item 
 $T_m^*(X^*)=\spn\{e_1,\dots,e_{2m}\}$ for all $m\ge1$.
\item $T_n^*(T_m^*(X^*))\subset T_m^*(X^*)$ for all $m,n\ge 1$.
\item $Y=\ell^1(\NN)=X$.
\item $Z=c_0(\NN)$, which is $w^*$-sequentially dense in $X^*=\ell^\infty(\NN)$.
\end{itemize}

If $\phi\in c_0(\NN)$, then we have
\[
\|T_n^*\phi-\phi\|_\infty
\le  \|S^{*n}\phi\|_\infty+\|P_{2n}\phi-\phi\|_\infty\to0 \quad(n\to\infty).
\]
Thus $T_n^*\phi\to\phi$ in norm for all $\phi\in Z$. Therefore (iii) holds.

However, if $\phi\in \ell^\infty(\NN)\setminus c_0(\NN)$,
then
\[
\langle e_1,\,T_n^*\phi-\phi\rangle=\langle T_ne_1-e_1,\,\phi\rangle
=\langle e_{n+1},\phi\rangle\not\to0 \quad(n\to\infty),
\]
and so $T_n^*\phi\not\to\phi$ weak*. Therefore (i) fails.
\end{remark}

Finally, the main result of the section,
Theorem~\ref{T:norm-norm}, follows by combining 
Lemmas~\ref{L:weak-weak*}, \ref{L:weak-norm} and \ref{L:weak*-norm}.


\section{Summability}\label{S:summability}

\subsection{The basic set-up}\label{S:setuo}

In the rest of the paper, 
we consider the following set-up. As before, $X$ denotes a Banach space with dual space $X^*$.
Let $(e_k)_{k\ge0}$ and $(\psi_k)_{k\ge0}$ be sequences in  $X$ and $X^*$ respectively such that
\[
\left\{
\begin{aligned}
\langle e_j,\psi_k\rangle&=0, \quad\forall j,k,~j\ne k,\\
\langle e_k,\psi_k\rangle&\ne0,\quad\forall k.
\end{aligned}
\right.
\]
For each $k\ge0$, define $P_k:X\to X$ by
\[
P_k:=\frac{e_k\otimes\psi_k}{\langle e_k,\psi_k\rangle}.
\]
Explicitly,
\[
P_k(x):=\frac{\langle x,\psi_k\rangle}{\langle e_k,\psi_k\rangle}e_k
\quad(x\in X).
\]
Clearly $P_k(e_k)=e_k$ and $P_k(e_j)=0$ if $j\ne k$. 
It is easy to see that $P_k$ is a rank-one projection ($P_k^2=P_k$) with
\[
\|P_k\|=\frac{\|e_k\|\|\psi_k\|}{|\langle e_k,\psi_k\rangle|}.
\]
Its adjoint $P_k^*:X^*\to X^*$ is given by
\[
P_k^*=\frac{\psi_k\otimes e_k}{\langle e_k,\psi_k\rangle}.
\]
Let $A=(a_{nk})_{n,k\ge0}$ be an infinite matrix of complex scalars such that, 
for each $n\ge0$,
\begin{equation}\label{E:convcond}
\sum_{k\ge 0} |a_{nk}|\|P_k\|<\infty.
\end{equation}
We define the associated summation operators $S_n^A:X\to X~(n\ge0)$
by the absolutely convergent series 
\[
S_n^A:=\sum_{k\ge0} a_{nk}P_k.
\]
Explicitly, we have
\begin{align*}
S_n^A(x)&=\sum_{k\ge 0} a_{nk}\frac{\langle x,\psi_k\rangle}{\langle e_k,\psi_k\rangle}e_k\quad(x\in X),\\
(S_n^A)^*(\phi)&=\sum_{k\ge 0} a_{nk}\frac{\langle e_k,\phi\rangle}{\langle e_k,\psi_k\rangle}\psi_k
\quad(\phi\in X^*).
\end{align*}


\subsection{The main equivalence}\label{S:equivalence}

With this notation established, our goal is to prove the following theorem.

\begin{theorem}\label{T:equivalence}
Let $Z$ be the norm-closure in $X^*$ of $\spn\{\psi_k:k\ge0\}$.
Suppose that there exists at least one matrix $A^0$ satisfying \eqref{E:convcond}
such that $S_n^{A^0}x\to x$ (either weakly or in norm) for all $x\in X$. 
Then, for every matrix $A$ satisfying  \eqref{E:convcond},
the following statements are equivalent:
\begin{enumerate}[\normalfont\rm(i)]
\item $S_n^A(x)\to x$ in $(X,w)$ for all $x\in X$;
\item $S_n^A(x)\to x$ in $(X,\|\cdot\|)$ for all $x\in X$;
\item $(S_n^A)^*(\phi)\to\phi$ in $(X^*,w^*)$ for all $\phi\in X^*$;
\item $(S_n^A)^*(\phi)\to\phi$ in $(X^*,w^*)$ for all $\phi\in Z$;
\item $(S_n^A)^*(\phi)\to\phi$ in $(X^*,\|\cdot\|)$ for all $\phi\in Z$.
\end{enumerate}
If, further $X$, is reflexive, then these are equivalent to:
\begin{enumerate}[\normalfont\rm(vi)]
\item $(S_n^A)^*(\phi)\to\phi$ in $(X^*,\|\cdot\|)$ for all $\phi\in X^*$.
\end{enumerate}
\end{theorem}

\begin{proof}
By \eqref{E:convcond}, for each $n\ge0$, there exists $K_n\ge0$ such that
\[
\sum_{k>K_n}|a_{nk}|\|P_k\|<2^{-n}.
\]
Define $T_n^A:X\to X$ by
\[
T_n^A:=\sum_{k=0}^{K_n} a_{nk}P_k.
\]
Clearly we have $\|S_n^A-T_n^A\|< 2^{-n}$ for all $n$,
so each of the statements (i)--(vi) holds iff it holds with $S_n^A$ replaced by $T_n^A$.

The operators $T_n^A$ are bounded, finite-rank operators on $X$. 
The images of $T_n^A$ and $(T_n^A)^*$ are given by
\begin{align}
T_n^A(X)&=\spn\{e_k: 1\le k\le K_n,~ a_{nk}\ne0\},\label{E:T_n^A}\\
(T_n^A)^*(X^*)&=\spn\{\psi_k: 1\le k\le K_n,~ a_{nk}\ne0\}.\label{E:T_n^A*}
\end{align}
It follows that $T_n^AT_m^A(X)\subset T_m^A(X)$ and $(T_n^A)^*(T_m^A)^*(X^*)\subset (T_m^A)^*(X^*)$ for all $m,n$.
Thus Theorem~\ref{T:norm-norm} applies. 

We claim that, for each $k$, there is at least one $n$ such that $a_{nk}\ne0$.
Indeed, if $a_{nk}=0$ for all~$n$, 
then $T_n^A(e_k)=0$ and $(T_n^A)^*(\psi_k)=0$ for all $n$, and so
(under the relevant assumption (i)--(vi)), 
either $e_k=0$ or $\psi_k=0$.
Neither of these  can be true, since we are assuming that $\langle e_k,\psi_k\rangle\ne0$.

It follows that, in the notation of Theorem~\ref{T:norm-norm},
\begin{align*}
Y&:=\overline{\spn(\cup_{m\ge0}(T_m^A)(X))}=\overline{\spn\{e_k:k\ge0\}},\\
Z&:=\overline{\spn(\cup_{m\ge0}(T_m^A)^*(X^*))}=\overline{\spn\{\psi_k:k\ge0\}}.
\end{align*}
In particular, this reconciles the definition of $Z$ given in the statement of Theorem~\ref{T:equivalence}
with that in Theorem~\ref{T:norm-norm}.

From Theorem~\ref{T:norm-norm}, applied with $A=A^0$,
we deduce that $Y=X$ and that $Z$ is $w^*$-sequentially dense in $X^*$.
Reapplying Theorem~\ref{T:norm-norm} with a general $A$ now gives the result.
\end{proof}

\begin{remarks}
(1) The implications (i)$\iff$(ii)$\iff$(iii)$\Longrightarrow$(iv)$\iff$(v) hold without the assumption about $A^0$.
The existence of $A^0$ is needed only for the implication (iv)$\Longrightarrow$(iii).

(2) In the concrete examples that we shall treat below,
the existence of $A^0$ will be an obvious consequence of some version of Fej\'er's theorem.

(3)
However, there do exist examples 	
where there exists no matrix $A^0$ satisfying~(ii). For instance, this is the case whenever
the Banach space $X$ fails to have the bounded approximation property (see \cite{MR19}). 
The article \cite{MPR22b} contains another such example, in which $X$ is actually a Hilbert space.
\end{remarks}


\subsection{Limitation theorems}\label{S:limitation}

Let $A=(a_{nk})_{n,k\ge0}$ be an infinite scalar matrix satisfying \eqref{E:convcond}.
We say that $A$ admits a left inverse if there exists a scalar matrix $B=(b_{jn})_{j,n\ge0}$  such that
\begin{equation}\label{E:inverse}
\sum_{n\ge0}|b_{jn}|<\infty \quad (j\ge0)
\qquad\text{and}\qquad
\sum_{n\ge0}b_{jn}a_{nk}=
\begin{cases}
1, &j=k,\\
0, &j\ne k.
\end{cases}
\end{equation}

\begin{theorem}\label{T:limitation}
Suppose that $S_n^A(x)\to x$ (weakly or in norm) for all $x\in X$, and that $A$ admits a left inverse $B$.
Then, writing $B_j:=\sum_{n\ge0}|b_{jn}|$, we have
\[
\frac{\|e_j\|\|\psi_j\|}{|\langle e_j,\psi_j\rangle|}=O(B_j) \quad(j\to\infty).
\]
\end{theorem}

\begin{remark}
This theorem is of interest when $\|e_j\|\|\psi_j\|/|\langle e_j,\psi_j\rangle|$ grows with $j$,
since it then places limitations on possible matrices $A$ for which summability holds.
Hardy  \cite[p.57]{Ha91} calls such results limitation theorems.
\end{remark}

\begin{proof}
For each $j\ge0$, we have
\begin{align*}
\sum_{n\ge0}b_{jn}S_n^A
&=\sum_{n\ge0}b_{jn}\Bigl(\sum_{k\ge0}a_{nk}P_k\Bigr)\\
&=\sum_{k\ge0}\Bigl(\sum_{n\ge0}b_{jn}a_{nk}\Bigr)P_k
=\sum_{k\ge0}\delta_{jk}P_k=P_j,
\end{align*}
the exchange of sums being justified by absolute convergence in \eqref{E:convcond} and \eqref{E:inverse}.
Now as $S_n^A(x)\to x$ for all $x\in X$, we have $\sup_n\|S_n^A(x)\|<\infty$ for each $x\in X$,
and hence, by the Banach--Steinhaus theorem, $\sup_n\|S_n^A\|<\infty$. It follows that
\[
\|P_j\|=\Bigl\|\sum_{n\ge0}b_{jn}S_n^A\Bigr\|\le \sum_{n\ge0}|b_{jn}|\|S_n^A\|\le B_j\sup_n\|S_n^A\|,
\]
in other words, $\|P_j\|=O(B_j)$. Finally, since $\|P_j\|=\|e_j\|\|\psi_j\|/|\langle e_j,\psi_j\rangle|$,
the result follows.
\end{proof}

In practice,  the matrix $B$ is not known explicitly, so estimating $B_j$ may be problematic.
Here is one case where we can do it.

We denote by $W^+(\DD)$ the holomorphic Wiener algebra, namely 
\[
W^+(\DD):=\Bigl\{f(z)=\sum_{k\ge0}\hat{f}(k)z^k: \sum_{k\ge0}|\hat{f}(k)|<\infty\Bigr\}.
\]

\begin{theorem}\label{T:inversebound}
Let $f\in W^+(\DD)$ with $f(0)\ne0$, and let $(\gamma_n)_{n\ge0}$ be  increasing sequence in $(0,\infty)$.
Let $A=(a_{nk})$ be the lower-triangular matrix with entries given by
\[
a_{nk}:=\gamma_n^{-1}\hat{(1/f)}(n-k) \quad(0\le k\le n).
\]
Then $A$ has a left inverse $B$, where $B=(b_{jn})$ is a lower-triangular matrix and
\[
\sum_{n\ge0}|b_{jn}|\asymp\gamma_j \quad(j\to\infty).
\]
\end{theorem}

\begin{proof}
Let $B$ be the lower-triangular matrix with entries $(b_{jn})$ given by
\[
b_{jn}:=\gamma_n \hat{f}(j-n) \quad(0\le n\le j).
\]
For each fixed $j$, there are only finitely many $n$ for which $b_{jn}\ne0$, so 
the first condition in \eqref{E:inverse} is clearly satisfied. As for the second condition,
we have
\begin{align*}
\sum_{n\ge0}b_{jn}a_{nk}
&=\sum_{n\ge0}\gamma_n\hat{f}(j-n)\gamma_n^{-1}\hat{(1/f)}(n-k)\\
&=\sum_{\substack{p,q\ge0\\p+q=j-k}}\hat{f}(p)\hat{(1/f)}(q)
=\hat{1}(j-k)=\delta_{jk}.
\end{align*}
Thus $B$ is indeed a left inverse of $A$. Further, we have
\[
\sum_{n\ge0}|b_{jn}|=\sum_{n=0}^j\gamma_n|\hat{f}(j-n)|\le \gamma_j\sum_{\ell\ge0}|\hat{f}(\ell)|
\]
and
\[
\sum_{n\ge0}|b_{jn}|=\sum_{n=0}^j\gamma_n|\hat{f}(j-n)|\ge\gamma_j|\hat{f}(0)|,
\]
By assumption, we have $\sum_{\ell\ge0}|\hat{f}(\ell)|<\infty$ and $|\hat{f}(0)|>0$. It follows that
$\sum_{n\ge0}|b_{jn}|\asymp\gamma_j$, as claimed.
\end{proof}

\begin{remark}
In this case, it is easy to see that in fact $B$ is a two-sided inverse of $A$.
\end{remark}

We illustrate these results by applying them to one particular family of summability methods, 
namely the Ces\`aro means. 
Given $\alpha>-1$, let $A=(a_{nk})_{n,k\ge0}$ be the lower-triangular matrix defined by
\begin{equation}\label{E:cesaromatrix}
a_{nk}:=\binom{n}{k}\Big/\binom{n+\alpha}{k} \quad(0\le k\le n).
\end{equation}
(As usual, binomial coefficients with non-integer arguments are defined using the Gamma function.)
With this choice of $A$, we write
\[
\sigma_n^\alpha (x):=S_n^A(x) \quad(x\in X).
\]
In particular, we write $s_n$ for $\sigma_n^0$ and $\sigma_n$ for $\sigma_n^1$.

We remark that, if $-1< \alpha\le\beta$, then $\sigma_n^\beta$ includes $\sigma_n^\alpha$ in the sense that,
if $\sigma_n^\alpha(x)\to x$ for some $x\in X$, then also $\sigma_n^\beta(x)\to x$. (For scalars this is well known;
for Banach spaces, it follows by  \cite[Theorem~5.1]{MPR22a}.)

The following limitation theorem is an abstract version of \cite[Theorem~46]{Ha91}.

\begin{theorem}\label{T:cesarolimitation}
Let $\alpha\ge0$. If $\sigma_n^\alpha(x)\to x$ (weakly or in norm) for all $x\in X$, then
\[
\frac{\|e_j\|\|\psi_j\|}{|\langle e_j,\psi_j\rangle|}=O(j^\alpha) \quad(j\to\infty).
\]
\end{theorem}

\begin{proof}
Let $A$ be the Ces\`aro matrix defined by \eqref{E:cesaromatrix}.
A computation gives
\[
a_{nk}=\binom{n}{k}\Big/\binom{n+\alpha}{k}=\binom{n-k+\alpha}{\alpha}\Bigl/\binom{n+\alpha}{\alpha}=\gamma_n^{-1}\hat{g}(n-k),
\]
where $(\gamma_n)$ is the increasing sequence given by
\[
\gamma_n:=\binom{n+\alpha}{\alpha},
\]
and where $\hat{g}(m)$ are the Taylor coefficients of the function 
\[
g(z):=\sum_{m=0}^\infty \binom{m+\alpha}{\alpha}z^m=(1-z)^{-\alpha-1}.
\]
Clearly $g=1/f$, where $f(z):=(1-z)^{\alpha+1}$. 
Since the Taylor coefficients of $f$ satisfy
\[
\sum_{k\ge1}|\hat{f}(k)|=
\sum_{k\ge1}\Bigl|\frac{\hat{(f')}(k-1)}{k}\Bigr|
\le \Bigl(\sum_{k\ge1}\frac{1}{k^2}\Bigr)^{1/2}\|f'\|_{H^2}<\infty,
\]
we have $f\in W^+(\DD)$, and Theorem~\ref{T:inversebound} applies. We deduce that $A$ has a lower-triangular inverse $B=(b_{jn})$ such that
\[
\sum_{n\ge0}|b_{jn}|=O(\gamma_j)=O\Bigl(\binom{j+\alpha}{\alpha}\Bigr)=O(j^\alpha) \quad(j\to\infty),
\]
the last equality by Stirling's formula. The result now follows by applying Theorem~\ref{T:limitation}.
\end{proof}


\section{Applications in spaces of continuous functions}\label{S:continuous}

\subsection{Fourier series}\label{S:Fourier}
Probably the best-known applications of summability are to Fourier series,
so, for our first example, we see what the abstract theory developed in the previous
two sections tells us about this case.

Let $\TT$ denote the unit circle. We write $C(\TT)$ for the space of complex-valued continuous
functions on $\TT$, with the usual sup-norm $\|f\|_\infty:=\sup_\TT|f|$.
The dual space of $C(\TT)$ may be identified with $M(\TT)$, 
the space of finite complex Borel measures on $\TT$, the duality being given by
\[
\langle f,\mu\rangle:=\int_\TT f(\zeta)\,d\mu(\overline{\zeta})
\quad(f\in C(\TT),~\mu\in M(\TT)).
\]
Under this pairing, $M(\TT)$ inherits the norm of $C(\TT)^*$, which is just the total-variation norm.

The  absolutely continuous measures on $\TT$ form a  closed subspace of $M(\TT)$, which can be identified with
$L^1(\TT)$ via $g\leftrightarrow g(e^{it})dt/2\pi$. The restriction of the total variation norm to $L^1(\TT)$
is just the usual $L^1$ norm.

We now proceed to apply the theory developed in \S\ref{S:summability}.
Let $X:=C(\TT)$, and for $k\in\ZZ$ let $e_k:=e^{ikt}$ and $\psi_k:=e^{ikt}\,dt/2\pi$.
(Here the index $k$ runs through all the integers rather than just the positive integers,
but this creates no problems.)
Then $\|e_k\|_\infty=\|\psi_k\|_1=1$ and $\langle e_k,\psi_k\rangle=1$ for all $k\in\ZZ$.

Let $A=(a_{nk})_{n\ge0, k\in\ZZ}$ be an infinite matrix of complex scalars
such that $\sum_{k\in\ZZ}|a_{nk}|<\infty$ for each~$n\ge0$. Then, for each $n\ge0$, we have
\begin{align*}
S_n^A(f)&=\sum_{k\in\ZZ}a_{nk}\frac{\langle f,\psi_k\rangle}{\langle e_k,\psi_k\rangle} e_k=\sum_{k\in\ZZ}a_{nk}\hat{f}(k)e^{ikt} \quad(f\in C(\TT)),\\
(S_n^A)^*(\mu)&=\sum_{k\in\ZZ}a_{nk}\frac{\langle e_k,\mu\rangle}{\langle e_k,\psi_k\rangle} \psi_k=\sum_{k\in\ZZ}a_{nk}\hat{\mu}(k)e^{ikt}\,\frac{dt}{2\pi} \quad(\mu\in M(\TT)).
\end{align*}

By Fej\'er's theorem, we have $\|\sigma_nf- f\|_\infty\to0$ as $n\to\infty$ for all $f\in C(\TT)$,
so there exists at least one matrix $A^0$ for which $\|S_n^{A^0}(f)-f\|_\infty\to0$. Thus 
Theorem~\ref{T:equivalence} applies. In the notation of Theorem~\ref{T:equivalence}, 
$Z$ is the norm-closure in $M(\TT)$ of $\spn\{e^{ikt} dt/2\pi:k\in\ZZ\}$, which is exactly $L^1(\TT)$. 
Thus we obtain the following theorem.

\begin{theorem}\label{T:Fourier}
Let $(a_{nk})_{n\ge0,k\in\ZZ}$ be an infinite matrix of complex scalars such that $\sum_{k\in\ZZ}|a_{nk}|<\infty$
for each $n\ge0$. Then the following statements are equivalent:
\begin{enumerate}[\normalfont\rm(i)]
\item $\sum_{k\in\ZZ}a_{nk}\hat{f}(k)e^{ikt} \to f$ in $(C(\TT),w)$, for all $f\in C(\TT)$;
\item $\sum_{k\in\ZZ}a_{nk}\hat{f}(k)e^{ikt}\to f$ in $(C(\TT),\|\cdot\|_\infty)$), for all $f\in C(\TT)$;
\item $\sum_{k\in\ZZ}a_{nk}\hat{\mu}(k)e^{ikt}dt/2\pi\to \mu$ in $(M(\TT),w^*)$, for all $\mu\in M(\TT)$;
\item $\sum_{k\in\ZZ}a_{nk}\hat{g}(k)e^{ikt}dt/2\pi\to g(e^{it})dt/2\pi$ in $(M(\TT),w^*)$ for all $g\in L^1(\TT)$;
\item $\sum_{k\in\ZZ}a_{nk}\hat{g}(k)e^{ikt}\to g$ in $(L^1(\TT),\|\cdot\|_1)$, for all $g\in L^1(\TT)$.
\end{enumerate}
\end{theorem}

We now use this theorem to deduce some classical results about Ces\`aro summation of Fourier series.
Since the summation index $k$ runs over $\ZZ$ rather than $\ZZ^+$, the definition of $\sigma_n^\alpha$ needs to be
modified accordingly, taking $a_{nk}:=\binom{n}{|k|}/\binom{n+\alpha}{|k|}$ for $|k|\le n$ and $a_{nk}:=0$ for $|k|>n$.  As usual, we  write $s_n$ for $\sigma_n^0$ and $\sigma_n$ for $\sigma_n^1$ (this was already implicit when we quoted
Fej\'er's theorem above).

\begin{theorem}\label{T:Fourierdiv}
\begin{enumerate}[\normalfont\rm(1)]
\item If $\mu=\delta_1$,  then $s_n(\mu)\not\to\mu$  in $(M(\TT),w^*)$.
\item There exists $f\in C(\TT)$ such that $s_n(f)\not\to f$ in $(C(\TT),w)$.
\item There exists $g\in L^1(\TT)$ such that $s_n(g)\not\to g$ in $(M(\TT),w^*)$.
\end{enumerate}
\end{theorem}

\begin{proof}
(1) If $\mu=\delta_1$, then $\hat{\mu}(k)=1$ for all $k\in\ZZ$, so
\[
s_n(\mu)=\sum_{k=-n}^n e^{ikt}\frac{dt}{2\pi}=D_n(t)\,\frac{dt}{2\pi},
\]
where $D_n(t)$ is the Dirichlet kernel. We know that
\[
\|s_n(\mu)\|_{M(\TT)}=\|D_n\|_1\asymp \log n\to\infty,
\]
so $(s_n(\mu))$ is not weak*-convergent. 

Parts (2) and (3) now follow from (1) by applying Theorem~\ref{T:Fourier}.
\end{proof}

\begin{remark}
Part (2) is a weak form of a famous result of du Bois-Reymond,
who showed that the Fourier series of a continuous function may diverge at a point.
\end{remark}

\begin{theorem}\label{T:Fourierconv} Let $\alpha>0$. 
\begin{enumerate}[\normalfont\rm(1)]
\item $\|\sigma_n^\alpha(f)-f\|_\infty\to0$ for all $f\in C(\TT)$.
\item $\|\sigma_n^\alpha(g)-g\|_1\to0$  for all $g\in L^1(\TT)$.
\item $\sigma_n^\alpha(\mu)\to\mu$ in $(M(\TT),w^*)$ for all $\mu\in M(\TT)$.
\end{enumerate}
\end{theorem}

\begin{proof}
Part~(1) is a classical result of M. Riesz \cite{Ri11}.
Parts (2) and (3) follow, using Theorem~\ref{T:Fourier}.
\end{proof}


\subsection{The disk algebra}\label{S:diskalg}

The disk algebra $A(\DD)$ consists of those holomorphic functions on 
the open unit disk $\DD$ that have a continuous extension to $\overline{\DD}$.
It is a Banach space (indeed a Banach algebra) with respect to the sup-norm.

By the maximum principle, the map $f\mapsto f|_\TT$ is an isometry
of $A(\DD)$ into $C(\TT)$, so $A(\DD)$ can be identified with a closed subspace $A$ of $C(\TT)$. 
In fact $A=\{f\in C(\TT): \hat{f}(k)=0 ~\forall k<0\}$.
Therefore the dual of $A(\DD)$ may be identified with the quotient $M(\TT)/A^\perp$, 
where 
\[
A^\perp:=\Bigl\{\mu\in M(\TT):\int_\TT f\,d\mu=0~ \forall f\in A\Bigr\}.
\]
By the F.~\&~M.~Riesz theorem, if $\mu\in A^\perp$,
then $\mu$ is absolutely continuous with respect to Lebesgue measure on $\TT$, say $\mu=h \,dt/2\pi$. The condition that 
$h\,dt/2\pi\in A^\perp$ is equivalent to $h\in \overline{H^1_0}$,
where
\[
\overline{H}_0^1:=\{h\in L^1(\TT): \hat{h}(k)=0 ~\forall k\ge0\}.
\]
Thus we can identify the dual space $A(\DD)^*$ with $M(\TT)/\overline{H^1_0}$.

There is another way to express this duality, using
Cauchy transforms. Given $\mu\in M(\TT)$, we define its
Cauchy transform $K\mu:\DD\to\CC$ by
\[
K\mu(z):=\int_\TT \frac{d\mu(\zeta)}{1-\overline{\zeta}z}
\quad(z\in \DD).
\]
Notice that $K\mu\equiv0\iff \mu\in \overline{H_0^1}$.
Hence the map $[\mu]\mapsto K\mu$ is a linear isomorphism
of $M(\TT)/\overline{H_0^1}$ onto $\cK$, where
\[
\cK:=\{K\mu:\mu\in M(\TT)\}.
\]
We endow $\cK$ with the norm that makes this isomorphism an isometry, namely
\[
\|K\mu\|_\cK:=\|[\mu]\|_{M(\TT)/\overline{H_0^1}}=\dist(\mu, \overline{H_0^1})
\quad(\mu\in M(\TT)).
\]
Thus, finally, the dual of $A(\DD)$ may be identified with $\cK$,
the duality pairing being given by
\[
\langle f,K\mu\rangle:=\int_\TT f(\zeta)\,d\mu(\overline{\zeta})
=\lim_{r\to1^-}\sum_{k=0}^\infty \hat{f}(k)\hat{\mu}(k)r^k
\quad(f\in A(\DD),~K\mu\in\cK).
\]

We now apply the theory developed in \S\ref{S:summability}.
Let $X:=A(\DD)$ and $X^*=\cK$.
For $k\ge0$, we define $e_k\in A(\DD)$ by $e_k(z):=z^k$, and $\psi_k\in\cK$ by
\[
\psi_k(z):=K\Bigl(e^{ikt}\,\frac{dt}{2\pi}\Bigr)(z)=z^k.
\]
It is easily checked that $\|e_k\|_\infty=\|\psi_k\|_\cK=1$ and that $\langle e_k,\psi_k\rangle=1$ for all~$k$.

Let $A=(a_{nk})_{n,k\ge0}$ be an infinite matrix of complex scalars
such that $\sum_{k\ge0}|a_{nk}|<\infty$ for each~$n\ge0$. Then, for each $n\ge0$, we have
\begin{align*}
S_n^A(f)&=\sum_{k\ge0}a_{nk}\frac{\langle f,\psi_k\rangle}{\langle e_k,\psi_k\rangle} e_k
=\sum_{k\ge0}a_{nk}\hat{f}(k)z^k \quad(f\in A(\DD)),\\
(S_n^A)^*(K\mu)&=\sum_{k\ge0}a_{nk}\frac{\langle e_k,K\mu\rangle}{\langle e_k,\psi_k\rangle} \psi_k
=\sum_{k\ge0}a_{nk}\hat{\mu}(k)z^k \quad(K\mu\in \cK).
\end{align*}

Just as in the case of $C(\TT)$, 
Fej\'er's theorem implies that $\|\sigma_nf- f\|_\infty\to0$ as $n\to\infty$ for all $f\in A(\DD)$,
so there exists at least one matrix $A^0$ for which $\|S_n^{A^0}(f)-f\|_\infty\to0$. Thus 
Theorem~\ref{T:equivalence} applies. In the notation of Theorem~\ref{T:equivalence}, 
$Z$ is the norm-closure in $\cK$ of $\spn\{K(e^{ikt} dt/2\pi):k\ge0\}$, which is  $L^1(\TT)/\overline{H^1_0}$. 
Thus we obtain the following theorem.

\begin{theorem}
Let $(a_{nk})_{n,k\ge0}$ be an infinite matrix of complex scalars such that $\sum_{k\in\ZZ}|a_{nk}|<\infty$
for each $n\ge0$. Then the following statements are equivalent:
\begin{enumerate}[\normalfont\rm(i)]
\item $\sum_{k\ge0}a_{nk}\hat{f}(k)z^k \to f$ in $(A(\DD),w)$, for all $f\in A(\DD)$;
\item $\sum_{k\ge0}a_{nk}\hat{f}(k)z^k\to f$ in $(A(\DD),\|\cdot\|_\infty)$ for all $f\in A(\DD)$;
\item $\sum_{k\ge0}a_{nk}\hat{\mu}(k)z^k\to K\mu$ in $(\cK,w^*)$, for all $\mu\in M(\TT)$;
\item $\sum_{k\ge0}a_{nk}\hat{g}(k)z^k\to Kg$ in $(\cK,w^*)$, for all $g\in L^1(\TT)$;
\item $\sum_{k\ge0}a_{nk}\hat{g}(k)z^k\to g$ in $(L^1/\overline{H^1_0},\|\cdot\|_{L^1/\overline{H^1_0}})$, for all $g\in L^1(\TT)$.
\end{enumerate}
\end{theorem}


\section{Application to Hardy spaces and Bergman spaces}\label{S:HardyBergman}

\subsection{Hardy spaces, $\bmoa$ and $\vmoa$}\label{S:Hardy}

We begin by reviewing the definitions and some basic facts about these spaces. 
All the details can be found in \cite[Chapter~9]{Zh07}.

For $1\le p<\infty$, the Hardy space $H^p$ is defined as the set of $f\in\hol(\DD)$ such that
\[
\|f\|_{H^p}:=\sup_{r<1}\Bigl(\frac{1}{2\pi}\int_0^{2\pi}|f(re^{i\theta})|^p\,d\theta\Bigr)^{1/p}<\infty.
\]
Also $H^\infty$ is  the set of bounded holomorphic functions on $\DD$, with 
\[
\|f\|_{H^\infty}:=\sup_\DD|f|.
\]

The space $\bmoa$ of holomorphic functions of bounded mean oscillation can be characterized as the 
space of $f\in H^2$ such that
\[
\|f\|_{\bmoa}:=|f(0)|+\sup_{\DD}(P|f|^2-|f|^2)^{1/2}<\infty,
\]
where $P|f|^2$ denotes the Poisson integral of $|f|^2$.
The space $\vmoa$ of holomorphic functions of vanishing mean oscillation
 is the closed subspace of $\bmoa$ consisting of those $f\in\bmoa$
such that $(P|f|^2-|f|^2)(z)\to0$ as $|z|\to1$.

All the spaces $H^p,\vmoa,\bmoa$ are Banach spaces and contain the polynomials.
Polynomials are dense in $H^p~(1\le p<\infty)$ and in $\vmoa$, but not in $H^\infty$ or $\bmoa$,
since neither of the latter is separable.

We have the following identifications of  dual spaces (up to isomorphism):  $(\vmoa)^*\cong H^1$ and $(H^1)^*\cong\bmoa$.
The pairings are given by
\begin{align*}
\langle f,g\rangle&:=\lim_{r\to1^-}\frac{1}{2\pi}\int_0^{2\pi}f(re^{i\theta})g(re^{-i\theta})\,d\theta
\quad(f\in \vmoa,~g\in H^1),\\
\langle g,h\rangle&:=\lim_{r\to1^-}\frac{1}{2\pi}\int_0^{2\pi}g(re^{i\theta})h(re^{-i\theta})\,d\theta
\quad(g\in H^1,~h\in\bmoa).
\end{align*}

We now proceed to apply the theory developed in \S\ref{S:summability}.
Let $X:=H^1$, let $e_k:=z^k\in H^1$ and let $\psi_k:=z^k\in \bmoa$.
Then $\|e_k\|=1$, $\|\psi_k\|\asymp1$ and $\langle e_k,\psi_k\rangle=1$ for all $k\ge0$.

Let $(a_{nk})_{n,k\ge0}$ be a matrix of complex scalars such that $\sum_k|a_{nk}|<\infty$ for each~$n$.
Then
\begin{align*}
S_n^A(g)&=\sum_{k\ge0}a_{nk}\frac{\langle g,\psi_k\rangle}{\langle e_k,\psi_k\rangle}e_k=\sum_{k\ge0}a_{nk}\hat{g}(k)z^k
\quad(g\in H^1),\\
(S_n^A)^*(h)&=\sum_{k\ge0}a_{nk}\frac{\langle e_k,h\rangle}{\langle e_k,\psi_k\rangle}\psi_k=\sum_{k\ge0}a_{nk}\hat{h}(k)z^k \quad(h\in \bmoa).
\end{align*}

By an appropriate version of Fej\'er's theorem, $\|\sigma_n(g)-g\|_{H^1}\to0$ for all $g\in H^1$.
Theorem~\ref{T:equivalence} therefore applies.
In the notation of Theorem~\ref{T:equivalence}, $Z$ is the norm-closure in $\bmoa$ of $\spn\{z^k: k\ge0\}$,
which is exactly $\vmoa$. 
We thus obtain the following result.

\begin{theorem}\label{T:BMOA}
Let $(a_{nk})_{n,k\ge0}$ be an infinite matrix of complex scalars such that $\sum_k|a_{nk}|<\infty$ for each~$n$.
Then the following statements are equivalent:
\begin{enumerate}[\normalfont\rm(i)]
\item $\sum_{k\ge0}a_{nk}\hat{g}(k)z^k\to g$ in $(H^1,w)$, for all $g\in H^1$;
\item $\sum_{k\ge0}a_{nk}\hat{g}(k)z^k \to g$ in $(H^1,\|\cdot\|_{H^1})$, for all $g\in H^1$;
\item $\sum_{k\ge0} a_{nk}\hat{h}(k)z^k\to h$ in $(\bmoa,w^*)$, for all $h\in \bmoa$;
\item $\sum_{k\ge0} a_{nk}\hat{f}(k)z^k\to f$ in $(\vmoa,w)$, for all $f\in \vmoa$;
\item $\sum_{k\ge0}a_{nk}\hat{f}(k)z^k\to f$ in $(\vmoa,\|\cdot\|_{\bmoa})$, for all $f\in \vmoa$.
\end{enumerate}
\end{theorem}

\begin{remark}
Explicitly, statement (iii) means that, for all $h\in \bmoa$ and all  $g\in H^1$,
\[
\lim_{n\to\infty}\lim_{r\to1^-}\sum_{k\ge0}a_{nk}r^k\hat{g}(k)\hat{h}(k)=
\lim_{r\to1^-}\frac{1}{2\pi}\int_0^{2\pi} g(re^{i\theta})h(re^{-i\theta})\,d\theta.
\]
\end{remark}

We now specialize to the case of Ces\`aro means. 

\begin{theorem}\label{T:BMOAdiv}
\begin{enumerate}[\normalfont\rm(1)]
\item There exists $g\in H^1$ such that $s_n(g)\not\to g$ in $(H^1,w)$.
\item There exists $f\in \vmoa$ such that $s_n(f)\not\to f$ in $(\vmoa,w)$.
\end{enumerate}
\end{theorem}

\begin{proof}
It is well known that there exists $g\in H^1$ such that $\|s_n(g)-g\|_{H1}\not\to0$.
This follows easily from the fact that the Riesz projection $P_+:L^1(\TT)\to H^1$ is unbounded
(see e.g.\ \cite[Ch.~III, \S1]{Ga07}).
Parts~(1) and (2) both follow by applying the equivalences in Theorem~\ref{T:BMOA}.
\end{proof}

\begin{remark}
Using much the same idea, 
Zhu has previously shown that there exists $f\in \vmoa$ such that $\|s_n(f)-f\|_{\bmoa}\not\to0$ (see \cite[Corollary~5]{Zh91}).
\end{remark}

\begin{theorem}\label{T:BMOAconv}
Let $\alpha>0$.
\begin{enumerate}[\normalfont\rm(1)]
\item $\|\sigma_n^\alpha(g)-g\|_{H^1}\to0$ for all $g\in H^1$.
\item $\sigma_n^\alpha(h)\to h$ in $(\bmoa,w^*)$ for all $h\in \bmoa$.
\item $\|\sigma_n^\alpha(f)-f\|_{\bmoa}\to0$ for all $f\in\vmoa$.
\end{enumerate}
\end{theorem}

\begin{proof}
Part~(1) is a classical result of Hardy \cite{Ha13}.
Parts~(2) and (3) follow, using Theorem~\ref{T:BMOA}.
\end{proof}


\subsection{Bergman and Bloch spaces}\label{S:Bergman}

Our development parallels that in \S\ref{S:Hardy}.
Once again, we begin by reviewing the definitions and some basic facts about these spaces. 
The details are in \cite[Chapters~4 and~5]{Zh07}.

For $1\le p<\infty$, the Bergman space $A^p$ is defined as the space of $f\in\hol(\DD)$ such that
\[
\|f\|_{A^p}:=\Bigl(\frac{1}{\pi}\int_\DD|f(z)|^p\,dA(z)\Bigr)^{1/p}<\infty,
\]
where $dA$ denotes area measure on $\DD$.

The Bloch space $\cB$ consists of those $f\in\hol(\DD)$ such that
\[
\|f\|_{\cB}:=|f(0)|+\sup_{\DD}(1-|z|^2)|f'(z)|<\infty.
\]
The little Bloch space $\cB_0$ 
 is the closed subspace of $\cB$ consisting of those functions $f\in\cB$
such that $(1-|z|^2)|f'(z)|\to0$ as $|z|\to1$.

All the spaces $A^p,\cB,\cB_0$ are Banach spaces and contain the polynomials.
Polynomials are dense in $A^p~(1\le p<\infty)$ and in $\cB_0$, but not in $\cB$,
since the latter is not separable.

We have the following identifications of  dual spaces (up to isomorphism):  $(\cB_0)^*\cong A^1$ and $(A^1)^*\cong\cB$.
The pairings are given by
\begin{align*}
\langle f,g\rangle&:=\lim_{r\to1^-}\frac{1}{\pi}\int_{|z|<r} f(z)g(\overline{z})\,dA(z)
\quad(f\in \cB_0,~g\in A^1),\\
\langle g,h\rangle&:=\lim_{r\to1^-}\frac{1}{\pi}\int_{|z|<r}g(z)h(\overline{z})\,dA(z)
\quad(g\in A^1,~h\in\cB).
\end{align*}

Once again, we apply the theory developed in \S\ref{S:summability}.
Let $X:=A^1$, let $e_k:=z^k\in A^1$ and let $\psi_k:=z^k\in \cB$.
Then $\|e_k\|_{A^1}\asymp1/(k+1)$ and $\|\psi_k\|_\cB\asymp1$,
and $|\langle e_k,\psi_k\rangle|\asymp1/(k+1)$,
where the implied constants are independent of $k$. 

Let $(a_{nk})_{n,k\ge0}$ be a matrix of complex scalars such that $\sum_k|a_{nk}|<\infty$ for each~$n$.
Then
\begin{align*}
S_n^A(g)&=\sum_{k\ge0}a_{nk}\frac{\langle g,\psi_k\rangle}{\langle e_k,\psi_k\rangle}e_k=\sum_{k\ge0}a_{nk}\hat{g}(k)z^k
\quad(g\in A^1),\\
(S_n^A)^*(h)&=\sum_{k\ge0}a_{nk}\frac{\langle e_k,h\rangle}{\langle e_k,\psi_k\rangle}\psi_k=\sum_{k\ge0}a_{nk}\hat{h}(k)z^k \quad(h\in \cB).
\end{align*}

It is known that $\|\sigma_n(g)-g\|_{A^1}\to0$ for all $g\in A^1$ (see Theorem~\ref{T:Blochconv}\,(i) below).
Theorem~\ref{T:equivalence} therefore applies.
In the notation of Theorem~\ref{T:equivalence}, 
$Z$ is the norm-closure in $\cB$ of $\spn\{z^k: k\ge0\}$,
which is exactly $\cB_0$. 
We thus obtain the following result.

\begin{theorem}\label{T:Bloch}
Let $(a_{nk})_{n,k\ge0}$ be an infinite matrix of complex scalars such that $\sum_k|a_{nk}|<\infty$ for each~$n$.
Then the following statements are equivalent:
\begin{enumerate}[\normalfont\rm(i)]
\item $\sum_{k\ge0}a_{nk}\hat{g}(k)z^k\to g$ in $(A^1,w)$, for all $g\in A^1$;
\item $\sum_{k\ge0}a_{nk}\hat{g}(k)z^k\to g$ in $(A^1,\|\cdot\|_{A^1})$, for all $g\in A^1$;
\item $\sum_{k\ge0} a_{nk}\hat{h}(k)z^k\to h$ in $(\cB,w^*)$, for all $h\in \cB$;
\item $\sum_{k\ge0} a_{nk}\hat{f}(k)z^k\to f$ in $(\cB_0,w)$, for all $f\in \cB_0$;
\item $\sum_{k\ge0}a_{nk}\hat{f}(k)z^k\to f$  in $(\cB_0,\|\cdot\|_\cB)$, for all $f\in \cB_0$.
\end{enumerate}
\end{theorem}

\begin{remark}
Explicitly, statement (iii) means that, for all $h\in \cB$ and all  $g\in A^1$,
\[
\lim_{n\to\infty}\lim_{r\to1^-}\sum_{k\ge0}a_{nk}r^k\hat{g}(k)\hat{h}(k)=
\lim_{r\to1^-}\frac{1}{\pi}\int_{|z|<r} g(z)h(\overline{z})\,dA(z).
\]
\end{remark}


\subsection{Relationship between Hardy and Bergman spaces}\label{S:relationship}

The following result describes the relationship between summability in $H^p$ and summability in $A^p$.

\begin{theorem}\label{T:HpAp}
Let $1\le p<\infty$.
Let $A:=(a_{nk})_{n,k\ge0}$ be an infinite matrix of complex scalars such that $\sum_k|a_{nk}|<\infty$ for each~$n$.
If $\|S_n^A(f)-f\|_{H^p}\to0$ for all $f\in H^p$, then $\|S_n^A(g)-g\|_{A^p}\to0$ for all $g\in A^p$.
\end{theorem}

For the proof, we need a  lemma.

\begin{lemma}
Let $1\le p<\infty$.
\begin{enumerate}[\normalfont\rm(1)]
\item We have $H^p\subset A^p$, and $\|f\|_{A^p}\le \|f\|_{H^p}$ for all $f\in H^p$.
\item Let $S:\hol(\DD)\to\hol(\DD)$ be a linear map such that  $(Sf)_r=S(f_r)$ for all $f\in \hol(\DD)$ and all $r\in(0,1)$.
If $\|Sf\|_{H^p}\le C\|f\|_{H^p}$ for all $f\in H^p$, then $\|Sg\|_{A^p}\le C\|g\|_{A^p}$ for all $g\in A^p$.
\end{enumerate}
\end{lemma}

\begin{proof}
(1) By Fubini's theorem, we have  the identity
\[
\|f\|_{A^p}^p=\int_0^1\|f_r\|_{H^p}^p\,2r\,dr.
\]
As $\|f_r\|_{H^p}\le \|f\|_{H_p}~\forall r\in(0,1)$, 
we obtain $\|f\|_{A^p}^p\le \int_0^1\|f\|_{H^p}^p\,2r\,dr=\|f\|_{H^p}^p$.

(2) Again by the above identity, if  $f\in A^p$, then
\begin{align*}
\|Sf\|_{A^p}^p
&=\int_0^1 \|(Sf)_r\|_{H^p}^p\,2r\,dr=\int_0^1 \|S(f_r)\|_{H^p}^p\,2r\,dr\\
&\le C^p\int_0^1\|f_r\|_{H^p}^p\,2r\,dr=C^p\|f\|_{A^p}^p.\qedhere
\end{align*}
\end{proof}

\begin{remark}
In fact a result of Hardy and Littlewood shows that we even have $H^p\subset A^{2p}$. For a simple proof of this, and
an example showing that $2p$ is sharp, see the article of Vukoti\'c \cite{Vu03}. However, we do not need this here.
\end{remark}

\begin{proof}[Proof of Theorem~\ref{T:HpAp}]
Assume that $\|S_n^A(f)-f\|_{H^p}\to0$ for all $f\in H^p$.
By part~(1) of the lemma, $\|S_n^A(f)-f\|_{A^p}\to0$ for all $f\in H^p$.
By part~(2) of the lemma, $\|S_n^A:A^p\to A^p\|\le\|S_n^A:H^p\to H^p\|$ for all $n$,
and by the Banach--Steinhaus theorem, $\sup_n\|S_n^A:H^p\to H^p\|<\infty$.
The standard density argument now gives that  $\|S_n^A(g)-g\|_{A^p}\to0$ for all $g\in A^p$.
\end{proof}

Once again, we finish the section by applying the above work to classical Ces\`aro means. 

\begin{theorem}\label{T:Blochdiv}
\begin{enumerate}[\normalfont\rm(1)]
\item If $1<p<\infty$, then $\|s_n(g)-g\|_{A^p}\to0$.
\item There exists $g\in A^1$ such that $s_n(g)\not\to g$ in $(A^1,w)$.
\item There exists $f\in \cB_0$ such that $s_n(f)\not\to f$ in $(\cB_0,w)$.
\end{enumerate}
\end{theorem}

\begin{proof}
Part~(1) is a consequence of the (well-known) corresponding result in $H^p$ and Theorem~\ref{T:HpAp}.
By a result of Zhu \cite[Theorem~9]{Zh91}, there exists $g\in A^1$ such that $\|s_n(g)-g\|_{A^1}\not\to0$.
Parts (2) and (3) both follow by applying the equivalences in Theorem~\ref{T:Bloch}.
\end{proof}

\begin{remark}
Zhu has previously shown that there exists $f\in\cB_0$
such that $\|s_n(f)-f\|_{\cB}\not\to0$ (see \cite[Corollary~11]{Zh91}). 
The same result had also  been obtained earlier by Anderson, Clunie and Pommerenke \cite{ACP74}, but with a slightly different identification of the dual of $\cB_0$ and the predual of $\cB$.
\end{remark}

\begin{theorem}\label{T:Blochconv}
Let $\alpha>0$.
\begin{enumerate}[\normalfont\rm(1)]
\item $\|\sigma_n^\alpha(g)-g\|_{A^1}\to0$ for all $g\in A^1$.
\item $\sigma_n^\alpha(h)\to h$ in $(\cB,w^*)$ for all $h\in \cB$.
\item $\|\sigma_n^\alpha(f)-f\|_{\cB}\to0$ for all $f\in\cB_0$.
\end{enumerate}
\end{theorem}

\begin{proof}
Part~(1) follows from the corresponding result for $H^1$, together with Theorem~\ref{T:HpAp}.
Parts~(2) and (3) are consequences of (1), using the equivalences in Theorem~\ref{T:Bloch}.
\end{proof}


\section{Applications in Hilbert  spaces}\label{S:Hilbert}

\subsection{Abstract set-up}\label{S:Hsetup}

In the case of a Hilbert space, 
we can repeat the analysis of \S\ref{S:summability} using the inner product
in place of the duality pairing. The fact that an inner product is sesquilinear rather than bilinear
leads to some slight differences between the two theories.

Throughout this section, $H$ denotes a complex Hilbert space 
with inner product $\langle\cdot,\cdot\rangle$.
Also, if $T$ is a bounded linear operator on $H$, then
$T^*$ denotes the Hilbert-space adjoint of $T$,
namely the unique operator on $H$ such that
\[
\langle Tg,h\rangle=\langle g,T^*h\rangle \quad(g,h\in H).
\]

The following theorem is the analogue of Theorem~\ref{T:norm-norm}. 
It is proved in just the same way, so we omit the details. 
Of course, since a Hilbert space is reflexive,
there is no distinction between weak- and weak*-convergence.

\begin{theorem}\label{T:Hnorm-norm}
Let $(T_n)_{n\ge1}$ be a sequence of bounded, finite-rank operators on $H$
such that
\[
T_nT_m(H)\subset T_m(H) 
\quad\text{and}\quad
T_n^*T_m^*(H)\subset T_m^*(H)
\qquad(m,n\ge0).
\]
Then the  following statements are equivalent:
\begin{enumerate}[\normalfont(i)]
\item $T_n h\to h$ weakly for all $h\in H$;
\item $T_n h\to h$ in norm for all $h\in H$;
\item $T_n^* h\to h$ weakly for all $h\in H$;
\item $T_n^* h\to h$ in norm for all $h\in H$.
\end{enumerate}
\end{theorem}

Now suppose that $(e_k)_{k\ge0}$ and $(f_k)_{k\ge0}$ are two sequences in  $H$  such that
\[
\left\{
\begin{aligned}
\langle e_j,f_k\rangle&=0, \quad\forall j,k,~j\ne k,\\
\langle e_k,f_k\rangle&\ne0,\quad\forall k.
\end{aligned}
\right.
\]
For each $k\ge0$, define $P_k:H\to H$ by
\[
P_k:=\frac{e_k\otimes f_k}{\langle e_k,f_k\rangle}.
\]
Explicitly,
\[
P_k(h):=\frac{\langle h,f_k\rangle}{\langle e_k,f_k\rangle}e_k
\quad(h\in H).
\]
Clearly $P_k(e_k)=e_k$ and $P_k(e_j)=0$ if $j\ne k$. 
It is easy to see that $P_k$ is a rank-one projection with
\[
\|P_k\|=\frac{\|e_k\|\|f_k\|}{|\langle e_k,f_k\rangle|}.
\]
Its adjoint $P_k^*:H\to H$ is given by
\[
P_k^*=\frac{f_k\otimes e_k}{\langle f_k,e_k\rangle}.
\]
Let $A=(a_{nk})_{n,k\ge0}$ be an infinite matrix of complex scalars such that, 
for each $n\ge0$,
\[
\sum_{k\ge 0} |a_{nk}|\|P_k\|<\infty.
\]
We define the associated summation operators $S_n^A:H\to H~(n\ge0)$
by the absolutely convergent series 
\[
S_n^A:=\sum_{k\ge0} a_{nk}P_k.
\]
Explicitly, we have
\begin{align*}
S_n^A(h)&=\sum_{k\ge 0} a_{nk}\frac{\langle h,f_k\rangle}{\langle e_k,f_k\rangle}e_k\quad(h\in H),\\
(S_n^A)^*(h)&=\sum_{k\ge 0} \overline{a}_{nk}\frac{\langle h,e_k\rangle}{\langle f_k,e_k\rangle}f_k
\quad(h\in H).
\end{align*}

With this notation established, we have the following equivalence theorem.
It follows from Theorem~\ref{T:Hnorm-norm} in just the same way that 
Theorem~\ref{T:equivalence} follows from Theorem~\ref{T:norm-norm}. 
We omit the details.

\begin{theorem}\label{T:Hequivalence}
The following statements are equivalent:
\begin{enumerate}[\normalfont\rm(i)]
\item $S_n^A(h)\to h$ weakly, for all $h\in H$.
\item $S_n^A(h)\to h$ in norm, for all $h\in H$.
\item $(S_n^A)^*(h)\to h$ weakly, for all $h\in H$.
\item $(S_n^A)^*(h)\to h$ in norm, for all $h\in H$.
\end{enumerate}
\end{theorem}

There are also  Hilbert-space versions of the limitation theorems,
Theorem~\ref{T:limitation} and Theorem~\ref{T:cesarolimitation}. 
The proofs are the same as before.

\begin{theorem}\label{T:Hlimitation}
Suppose that $S_n^A(x)\to x$ (weakly or in norm) for all $x\in H$, and that $A$ admits a left inverse $B$.
Then, writing $B_j:=\sum_{n\ge0}|b_{jn}|$, we have
\[
\frac{\|e_j\|\|f_j\|}{|\langle e_j,f_j\rangle|}=O(B_j) \quad(j\to\infty).
\]
\end{theorem}

\begin{theorem}\label{T:Hcesarolimitation}
Let $\alpha\ge0$. If $\sigma_n^\alpha(x)\to x$ (weakly or in norm) for all $x\in H$, then
\[
\frac{\|e_j\|\|f_j\|}{|\langle e_j,f_j\rangle|}=O(j^\alpha) \quad(j\to\infty).
\]
\end{theorem}


\subsection{Reproducing kernel spaces of holomorphic functions}\label{S:RKHS}
We shall now apply these ideas to reproducing kernel  spaces of holomorphic functions
on the unit disk.

Let $H$ be a Hilbert space of holomorphic functions on $\DD$
such that:
\begin{itemize}
\item convergence in the norm of $H$ implies pointwise convergence on $\DD$;
\item $H$ contains the polynomials.
\end{itemize}
The first assumption implies that, for each $w\in\DD$, the functional
$h\mapsto h(w)$ is continuous on $H$, so, by the Riesz representation theorem, 
there exists a unique $k_w\in H$ such that
\[
h(w)=\langle h,k_w\rangle \quad(h\in H).
\]
We define $K: H\times H\to\CC$ by
\[
K(z,w):=k_w(z)=\langle k_w,k_z\rangle \quad (z,w\in\DD).
\]
The function  $K$ is  the reproducing kernel of $H$. Clearly it satisfies
$K(w,z)=\overline{K(z,w)}$, and $K(z,w)$ is holomorphic in $z$ for each fixed~$w$.
Therefore it is anti-holomorphic in $w$ for each fixed~$z$.
We shall need the following simple lemma about derivatives of $K$.

\begin{lemma}\label{L:RKHS}
For each $w\in\DD$ and each $n\ge0$, let
\[
k_{w,n}(z):=\frac{1}{n!}\frac{\partial^n}{\partial\overline{w}^n}K(z,w) \quad(z\in\DD).
\]
Then $k_{w,n}\in H$ and 
\begin{equation}\label{E:RKHS}
\langle h,k_{w,n}\rangle=\frac{h^{(n)}(w)}{n!}
\quad(h\in H).
\end{equation}
\end{lemma}

\begin{proof}
As norm-convergence in $H$ implies local uniform convergence
on $\DD$, the map $h\mapsto h^{(n)}(w)/n!$ is a continuous linear functional on $H$, 
so, by the Riesz theorem again, there exists 
 $k_{w,n}\in H$ such that \eqref{E:RKHS} holds.
It remains to identify $k_{w,n}$, which we do as follows. 
For each $z\in\DD$, we have
\begin{align*}
k_{w,n}(z)
&=\langle k_{w,n},k_z\rangle
=\overline{\langle k_z,k_{w,n}\rangle}\\
&=\overline{\frac{k_z^{(n)}(w)}{n!}}
=\overline{\frac{1}{n!}\frac{\partial^n}{\partial w^n}K(w,z)}
=\frac{1}{n!}\frac{\partial^n}{\partial\overline{w}^n}K(z,w).\qedhere
\end{align*}
\end{proof}

Now we define sequences $(e_j)_{j\ge0}$
and $(f_j)_{j\ge0}$ in $H$  by
$e_j:=z^j$ and $f_j:=k_{0,j}$.
Note that, if $h\in H$, then
\[
\langle h, f_j\rangle=\frac{h^{(j)}(0)}{j!}=\hat{h}(j).
\]
In particular,
$\langle e_i,f_j\rangle=\delta_{ij}$,
so the theory outlined in \S\ref{S:Hsetup} applies. Note also that $\|e_j\|=\|z^j\|_H$ and
\[
\|f_j\|^2=\|k_{0,j}\|^2=\langle k_{0,j},k_{0,j}\rangle=\frac{1}{j!}k_{0,j}^{(j)}(0)
=\frac{1}{j!^2}\frac{\partial^{2j}K}{\partial z^j\partial\overline{w}^j}(0,0).
\]
Theorem~\ref{T:Hequivalence} leads to the following result.

\begin{theorem}\label{T:RKHS}
Let $(a_{nj})_{n,j\ge0}$ be an infinite matrix of complex scalars such that
$\sum_{j\ge0}|a_{nj}|\|z^j\|\|k_{0,j}\|<\infty$ for each $n\ge0$.
Then the following statements are equivalent:
\begin{enumerate}[\normalfont\rm(i)]
\item $\sum_{j\ge0}a_{nj}\hat{h}(j)z^j\to h$ weakly as $n\to\infty$, for all $h\in H$;
\item $\sum_{j\ge0}a_{nj}\hat{h}(j)z^j\to h$ in norm  as $n\to\infty$, for all $h\in H$;
\item $\sum_{j\ge0}\overline{a}_{nj}\langle h,z^j\rangle k_{0,j}\to h$ weakly as $n\to\infty$, for all $h\in H$;
\item $\sum_{j\ge0}\overline{a}_{nj}\langle h,z^j\rangle k_{0,j}\to h$ in norm as $n\to\infty$, for all $h\in H$.
\end{enumerate}
\end{theorem}



The following result is a limitation theorem for reproducing kernel spaces.
It is an immediate consequence of Theorem~\ref{T:Hcesarolimitation}.

\begin{theorem}\label{T:RKHSlimitation}
Let $H$ be a reproducing-kernel space of holomorphic functions on $\DD$ that contains the polynomials.
If $\alpha\ge0$ and $\sigma_n^\alpha(h)\to h$ (weakly or  in norm) for all $h\in H$, then 
\[
\|z^j\|\|k_{0,j}\|=O(j^\alpha)\quad(j\to\infty).
\]
\end{theorem}


\subsection{De Branges--Rovnyak spaces}\label{S:H(b)}
We now specialize to the case where $H=\cH(b)$, the de Branges--Rovnyak space with symbol $b$.
Here $b$ is an element of the unit ball of $H^\infty$. By definition, $\cH(b)$ is the reproducing-kernel space on $\DD$
with  kernel
\[
K(z,w)=\frac{1-b(z)\overline{b(w)}}{1-z\overline{w}} \quad(z,w\in\DD).
\]
The space $\cH(b)$ contains the polynomials iff $b$ is a non-extreme point of the unit ball
of $H^\infty$, and in this case the polynomials are dense in $\cH(b)$. It is known that $b$ is non-extreme iff its boundary values satisfy $\log(1-|b|^2)\in L^1(\TT)$. Henceforth, we assume that this is the case. 
For further information on $\cH(b)$-spaces, we refer to \cite{Sa94} and \cite{FM16a,FM16b}.

The following result will be useful in what follows.

\begin{proposition}
Let $b$ be a non-extreme point of the unit ball of $H^\infty$.
Then $\inf_{j\ge0}\|k_{0,j}\|_{\cH(b)}>0$.
\end{proposition}

\begin{proof}
We compute an expression for $k_{0,j}$ in $\cH(b)$. Using Leibniz' theorem, we have
\begin{align*}
k_{0,j}(z)
&=\frac{1}{j!}\frac{\partial^j}{\partial \overline{w}^j}K(z,0)\\
&=\frac{1}{j!}\sum_{i=0}^j \binom{j}{i}\frac{\partial^i}{\partial \overline{w}^i}\Bigl(1-b(z)\overline{b(w)}\Bigr)\Bigl|_{w=0}\frac{\partial^{j-i}}{\partial \overline{w}^{j-i}}\Bigl(1-z\overline{w}\Bigr)^{-1}\Bigl|_{w=0}\\
&=z^j-b(z)\sum_{i=0}^j \overline{\hat{b}(i)}z^{j-i}.
\end{align*}
Thus, viewed as  functions on $\TT$,
\[
k_{0,j}=z^j(1-b\overline{s_j(b)}).
\]

Also, expanding $K(z,w)=(1-b(z)\overline{b(w)})/(1-z\overline{w})$ as a double power series in $z,\overline{w}$,
and computing the coefficient of $z^j\overline{w}^j$, we find that
\[
\|k_{0,j}\|_{\cH(b)}^2=\frac{1}{j!^2}\frac{\partial^{2j}K}{\partial z^j\partial\overline{w}^j}(0,0)
=1-\sum_{i=0}^j |\hat{b}(i)|^2=1-\|s_j(b)\|_{H^2}^2.
\]
In particular, we have
\[
1\ge \|k_{0,j}\|_{\cH(b)}^2\ge 1-\|b\|_{H^2}^2 \quad(j\ge0).
\]
As $b$ is non-extreme, $\|b\|_{H^2}<\|b\|_{H^\infty}\le1$ or $\|b\|_{H^2}=\|b\|_{H^\infty}<1$. 
Either way, $1-\|b\|_{H^2}^2>0$. Hence $\inf_{j\ge0}\|k_{0,j}\|_{\cH(b)}>0$, as claimed.
\end{proof}

\begin{remark}
A similar result holds in the more general finite-rank $\cH[B]$-spaces studied 
by Aleman and Malman in \cite{AM19}. 
However, not all reproducing kernel function spaces have this property. 
For example, in the classical Dirichlet space on the unit disk, 
we have $\|k_{0,j}\|\asymp 1/j$. 
\end{remark}

Feeding this  information into the results of \S\ref{S:RKHS}, we deduce the following results.
(We write $\langle\cdot,\cdot\rangle_b$ for the inner product in $\cH(b)$.)

\begin{theorem}\label{T:H(b)}
Let $b$ be a non-extreme point of the unit ball of $H^\infty$.
Let $(a_{nj})_{n,j\ge0}$ be a matrix of complex scalars such that
$\sum_{j\ge0}|a_{nj}|\|z^j\|_{\cH(b)}<\infty$
for each $n\ge0$.
Then the following statements are equivalent:
\begin{enumerate}[\normalfont\rm(i)]
\item $\sum_{j\ge0}a_{nj}\hat{h}(j)z^j\to h$ weakly in $\cH(b)$, for all $h\in\cH(b)$;
\item $\sum_{j\ge0}a_{nj}\hat{h}(j)z^j\to h$ in norm in $\cH(b)$, for all $h\in\cH(b)$;
\item $\sum_{j\ge0}\overline{a}_{nj}(1-b\overline{s_j(b)})\langle h,z^j\rangle_b z^j\to h$ weakly in $\cH(b)$, for all $h\in\cH(b)$;
\item $\sum_{j\ge0}\overline{a}_{nj}(1-b\overline{s_j(b)})\langle h,z^j\rangle_b z^j\to h$ in norm in $\cH(b)$, for all $h\in\cH(b)$.
\end{enumerate}
\end{theorem}


\begin{theorem}\label{T:deBrangeslimitation}
Let $b$ be a non-extreme point of the unit ball of $H^\infty$.
If $\alpha\ge0$ and $\sigma_n^\alpha(h)\to h$ (weakly or in norm)
for all $h\in \cH(b)$, then 
\[
\|z^j\|_{\cH(b)}=O(j^\alpha)\quad(j\to\infty).
\]
\end{theorem}

The following consequence is worth pointing out explicitly.

\begin{corollary}
Let $b$ be a non-extreme point of the unit ball of $H^\infty$.
If $s_n(h)\to h$ (weakly or in norm) for all $h\in \cH(b)$, then $\sup_{j\ge0}\|z^j\|_{\cH(b)}<\infty$.
\end{corollary}

The de Branges--Rovnyak spaces in which the powers of $z$ are bounded in norm form
an interesting class in their own right. For example, they are precisely the $\cH(b)$-spaces that
contain $H^\infty$. Also, they are characterized by the condition that $1/(1-|b|^2)\in L^1(\TT)$.
For more on this, see \cite[\S4]{Sa86}. 

For $\alpha>0$, the convergence of $\sigma_n^\alpha(h)$  to $h$ for every $h\in\cH(b)$ also
has implications for  $b$. To derive these, we need an explicit formula for $\|z^j\|_{\cH(b)}$,
which requires that we delve a little further into the theory
of de Branges--Rovnyak spaces. The details can be found in \cite{Sa94} and \cite{FM16b}. 

As mentioned earlier, if $b$ is a non-extreme point of the unit ball of $H^\infty$,
then $\log(1-|b|^2)\in L^1(\TT)$. This implies that
there exists a unique outer function $a$ on $\DD$ with $a(0)>0$
such that $|b|^2+|a|^2=1$ a.e.\ on $\TT$.
This function $a$ is sometimes called the
pythagorean complement of $b$. 
Writing $\phi:=b/a$, we obtain a function in the Smirnov class $N^+$,
namely the space of quotients of $H^\infty$-functions with outer denominators.
Conversely, all Smirnov functions are obtained in this way. There is thus 
a one-to-one correspondence $b\leftrightarrow \phi$,
between non-extreme points $b$ of the unit ball of $H^\infty$ and functions $\phi\in N^+$.
Expanding $\phi(z)$ as a Taylor series, say $\phi(z)=\sum_{j\ge0}c_jz^j$, we have
\begin{equation}\label{E:z^j}
\|z^j\|_{\cH(b)}^2=1+\sum_{i=0}^j |c_i|^2.
\end{equation}
A simple proof of \eqref{E:z^j} can be found for example in \cite[p.81]{Sa86}.

\begin{theorem}\label{T:phigrowth}
Let $b$ be a non-extreme point of the unit ball of $H^\infty$ and let $\phi$
be the corresponding function in the Smirnov class.
If $\alpha\ge0$ and $\sigma_n^\alpha(h)\to h$ (weakly or in norm)
for all $h\in \cH(b)$, then 
\begin{equation}\label{E:phi}
\phi(z)=O\Bigl((1-|z|)^{-\alpha-1/2}\Bigr) \quad(|z|\to1^-).
\end{equation}
\end{theorem}

\begin{proof}
As above, let us write  $\phi(z)=\sum_{j\ge0}c_jz^j$.
Summing by parts, we have
\[
|\phi(z)|\le \sum_{j\ge0}|c_j||z|^j=(1-|z|)\sum_{j\ge0}\Bigl(\sum_{i=0}^j|c_i|\Bigr)|z|^j \quad(|z|<1).
\]
By Cauchy--Schwarz and formula \eqref{E:z^j},
\[
\sum_{i=0}^j|c_i|\le (j+1)^{1/2}\Bigl(\sum_{i=0}^j|c_i|^2\Bigr)^{1/2}\le(j+1)^{1/2}\|z^j\|_{\cH(b)} 
\quad(j\ge0).
\]
From Theorem~\ref{T:deBrangeslimitation}, it follows that 
\[
\sum_{i=0}^j |c_i|=O(j^{\alpha+1/2})\quad(j\to\infty).
\]
Feeding this information back into the inequality for $\phi$, we find that there is a constant $C$ such that
\[
|\phi(z)|\le C(1-|z|)\sum_{j\ge0}j^{\alpha+1/2}|z|^j \quad(|z|<1).
\]
The conclusion \eqref{E:phi} follows.
\end{proof}

We conclude with two concrete examples.

\begin{corollary}\label{C:locD}
Let $\phi(z):=z/(1-z)$
and let $b$ be the corresponding non-extreme point of the unit ball of $H^\infty$.
If $\sigma^\alpha_n(h)\to h$ for all $h\in\cH(b)$, then $\alpha\ge 1/2$.
\end{corollary}

\begin{remark}
In this case $b$ can be determined explicitly, namely as
\[
b(z)=\frac{(1-w_0)z}{1-w_0z},
\]
where $w_0:=(3-\sqrt{5})/2$.
This example is taken from \cite{Sa97}, where it is shown that, with this particular choice of $b$,
the de Branges--Rovynak space $\cH(b)$ is equal to the local Dirichlet space $\cD(\delta_1)$
(with the same norm). 
It was shown in \cite{MPR21} that, if $\alpha>1/2$, 
then $\sigma_n^\alpha(h)\to h$ in norm for all $h\in\cD(\delta_1)$,
but that there exists $h_0\in\cD(\delta_1)$ such that $\sigma_n^{1/2}(h_0)\not\to h_0$.
Thus the constant $1/2$ in Corollary~\ref{C:locD} is sharp.
\end{remark}

\begin{corollary}
Let 
\[
\phi(z):=\exp\Bigl(\sqrt{\frac{1+z}{1-z}}\Bigr) \quad(z\in\DD).
\]
Then $\phi\in N^+$. If $b$ is the corresponding non-extreme point of the unit ball of $H^\infty$,
then there exists no $\alpha\ge0$ such that $\sigma_n^\alpha(h)\to h$ for all $h\in\cH(b)$.
\end{corollary}

\begin{proof}
The function $\sqrt{(1+z)/(1-z)}$ belongs to $H^1$.
Its exponential is therefore an outer function,
so $\phi\in N^+$.
Now apply Theorem~\ref{T:phigrowth}.
\end{proof}

\begin{remark}
Thus, in this case, even though polynomials are dense in $\cH(b)$,
there is no Ces\`aro summability method that always converges.
Such examples were previously obtained in \cite{EFKMR16} and \cite{MR18}
using inductive constructions. Our method has the virtue that it yields a
(fairly simple) explicit function $\phi$ that does the trick.
\end{remark}


\subsection*{Acknowledgements}
We 
thank Ryan Gibara for helpful discussions.

\bibliographystyle{plain}      
\bibliography{bibliography}

\end{document}